\documentclass[11pt]{article}%
\usepackage{amsmath}
\usepackage{amsfonts}
\usepackage{amssymb}
\usepackage{graphicx}%
\providecommand{\U}[1]{\protect\rule{.1in}{.1in}}
\newtheorem{theorem}{Theorem}

\newtheorem{corollary}[theorem]{Corollary}

\newtheorem{example}[theorem]{Example}

\newtheorem{lemma}[theorem]{Lemma}

\newtheorem{proposition}[theorem]{Proposition}
\newtheorem{remark}[theorem]{Remark}

\newenvironment{proof}[1][Proof]{\noindent\textbf{#1.} }{\ \rule{0.5em}{0.5em}}
\begin{document}

\title{Efficiency of the convex hull of the columns of certain triple perturbed
consistent matrices}
\author{Susana Furtado\thanks{Email: sbf@fep.up.pt. The work of this author was
supported by FCT- Funda\c{c}\~{a}o para a Ci\^{e}ncia e Tecnologia, under
project UIDB/04561/2020.} \thanks{Corresponding author.}\\CMAFcIO and Faculdade de Economia \\Universidade do Porto\\Rua Dr. Roberto Frias\\4200-464 Porto, Portugal
\and Charles R. Johnson \thanks{Email: crjohn@wm.edu. }\\Department of Mathematics\\College of William and Mary\\Williamsburg, VA 23187-8795}
\maketitle

\begin{abstract}
In decision making a weight vector is often obtained from a reciprocal matrix
$A$ that gives pairwise comparisons among $n$ alternatives. The weight vector
should be chosen from among efficient vectors for $A.$ Since the reciprocal
matrix is usually not consistent, there is no unique way of obtaining such a
vector. It is known that all weighted geometric means of the columns of $A$
are efficient for $A.$ In particular, any column and the standard geometric
mean of the columns are efficient, the latter being an often used weight
vector. Here we focus on the study of the efficiency of the vectors in the
(algebraic) convex hull of the columns of $A.$ This set contains the (right)
Perron eigenvector of $A,$ a classical proposal for the weight vector, and the
Perron eigenvector of $AA^{T}$ (the right singular vector of $A$), recently
proposed as an alternative. We consider reciprocal matrices $A$ obtained from
a consistent matrix $C$ by modifying at most three pairs of reciprocal entries
contained in a $4$-by-$4$ principal submatrix of $C$. For such matrices, we
give necessary and sufficient conditions for all vectors in the convex hull of
the columns to be efficient. In particular, this generalizes the known
sufficient conditions for the efficiency of the Perron vector. Numerical
examples comparing the performance of efficient convex combinations of the
columns and weighted geometric means of the columns are provided.

\end{abstract}

\bigskip



\section{Introduction}

An $n$-by-$n$ entry-wise positive matrix $A=\left[  a_{ij}\right]  $ is called
a \emph{reciprocal matrix }or\emph{ a pairwise comparison matrix} if
$a_{ji}=\frac{1}{a_{ij}},$ for all $1\leq i,j\leq n.$ We denote by
$\mathcal{PC}_{n}$ the set of all such matrices. Matrix $A$ is said to be
\emph{consistent} if $a_{ij}a_{jk}=a_{ik}$ for all $i,j,k.$ This happens if
and only if there is a positive vector $w=\left[
\begin{array}
[c]{ccc}%
w_{1} & \ldots & w_{n}%
\end{array}
\right]  ^{T}$ such that $a_{ij}=\frac{w_{i}}{w_{j}}$ for all $i,j.$ Such a
vector is unique up to a factor of scale. Any matrix in $\mathcal{PC}_{2}$ is consistent.

In the Analytic Hierarchy process, a method introduced by Saaty
\cite{saaty1977, Saaty} and used in decision analysis, reciprocal matrices
arise as independent, pairwise ratio comparisons among $n$ alternatives. If
the reciprocal matrix $A$ is consistent, $A=\left[  \frac{w_{i}}{w_{j}%
}\right]  ,$ the vector $w$ cardinally ranks the alternatives. However, when
$n>2,$ consistency of the ratio comparisons is unlikely. Thus, a cardinal
ranking vector should be obtained from a reciprocal matrix \cite{anh, baj,
choo, dij, golany} and it is natural to choose it from among efficient ones.

A positive vector $w=\left[
\begin{array}
[c]{ccc}%
w_{1} & \ldots & w_{n}%
\end{array}
\right]  ^{T}$ is called \emph{efficient} for $A\in\mathcal{PC}_{n}$
\cite{blanq2006} if, for every other positive vector $v=\left[
\begin{array}
[c]{ccc}%
v_{1} & \ldots & v_{n}%
\end{array}
\right]  ^{T},$
\[
\left\vert a_{ij}-\frac{v_{i}}{v_{j}}\right\vert \leq\left\vert a_{ij}%
-\frac{w_{i}}{w_{j}}\right\vert \text{ for all }1\leq i,j\leq n
\]
implies that $v$ and $w$ are proportional, i.e. no other consistent matrix
approximating $A$ is clearly better than the one associated with $w$ (Pareto
optimality). We denote the set of all efficient vectors for $A\ $ by
$\mathcal{E}(A).$ The set $\mathcal{E}(A)$ is connected \cite{blanq2006,FJ2}
and, when the $0$ vector is added, it is closed \cite{FJ1}.

The efficient vectors for a consistent matrix, say $A=\left[  \frac{w_{i}%
}{w_{j}}\right]  ,$ are the positive multiples of $w,$ that is, of any of its
columns. When $A$ is not consistent, there are infinitely many
(non-proportional) efficient vectors for $A.$ It is known that a vector $w$ is
efficient for $A$ if and only if a certain directed graph (digraph) $G(A,w),$
constructed from $A$ and $w,$ is strongly connected \cite{blanq2006, FJ3}.
Many ways of constructing efficient vectors for $A$ have been proposed.
Recently, a method to generate inductively all efficient vectors for a
reciprocal matrix was provided \cite{FJ2}. Also, a description of the set of
efficient vectors of a reciprocal matrix $A\in\mathcal{PC}_{n}\ $as a union of
at most $\frac{(n-1)!}{2}$ convex sets was given in \cite{FJ6}.

The classical proposal for the weight vector obtained from a reciprocal matrix
$A\in\mathcal{PC}_{n}$ is its (right) Perron eigenvector \cite{saaty1977,
Saaty}. For $n>3$ examples of reciprocal matrices for which the Perron vector
is inefficient are known \cite{blanq2006, bozoki2014, FJ5}. Also, classes of
reciprocal matrices for which the Perron vector is efficient have been
identified, with the pioneering works \cite{p6, p2}. In the sequence of these
results, all the efficient vectors in some classes of reciprocal matrices have
been explicitly described. More precisely, in \cite{CFF} and \cite{Fu22} a
description of the efficient vectors for a reciprocal matrix obtained from a
consistent matrix by perturbing one pair or two pairs of reciprocal entries
has been presented. These matrices were called \emph{simple} and \emph{double
perturbed consistent matrices} in \cite{p6} and \cite{p2}, respectively, where
the efficiency of their Perron vector was shown. In \cite{FJ6} the efficient
vectors for a reciprocal matrix obtained from a consistent one by modifying
one column (and the corresponding reciprocal row) were described. In
\cite{FJ4} an $s$-block perturbed consistent matrix was defined as a
reciprocal matrix obtained from a consistent one by modifying an $s$-by-$s$
principal submatrix ($s\geq2$)$.$ When $s=2,$ this class coincides with the
class of simple perturbed consistent matrices; when $s=3$ it includes a type
of double perturbed consistent matrices. In that paper we gave an explicit way
of constructing the efficient vectors for a $3$-block perturbed consistent
matrix from the efficient vectors for a $4$-by-$4$ principal submatrix that
contains the modified $3$-by-$3$ principal block. The class of $4$-block
perturbed consistent matrices includes the $s$-block perturbed consistent
matrices, for $s=2,3,$ and all types of double perturbed consistent matrices.
Classes of efficient vectors for $s$-block perturbed consistent matrices with
$s\geq4$ were also given in \cite{FJ4}. Any matrix in $\mathcal{PC}_{3}$ is a
$2$-block perturbed consistent matrix and any matrix in $\mathcal{PC}_{4}$ is
a $3$-block perturbed consistent matrix.

In \cite{FJ3} we have shown the efficiency for $A\in\mathcal{PC}_{n}\ $of the
vectors in the geometric convex hull of the columns of $A$ (the set of the
Hadamard products of entry-wise powers of the columns with nonnegative
exponents summing to $1$). In particular, the (Hadamard) geometric mean of all
the columns of $A$ \cite{blanq2006, fichtner86}, or of any subset of the
columns \cite{FJ1}, is efficient. The former is often used as the weight vector.

A natural question is the study of the efficiency of the vectors in the
(algebraic) convex hull of the columns of $A\in\mathcal{PC}_{n}.$ Note that,
since any positive multiple of an efficient vector is still efficient, when
all vectors in the convex hull of the columns of $A$ are efficient then the
cone generated by the columns (not including the $0$ vector) is contained in
$\mathcal{E}(A).$ We denote it by $\mathcal{C}(A).$ When $\mathcal{E}(A)$ is
convex then $\mathcal{C}(A)\subseteq\mathcal{E}(A)$, as any column of $A$ is efficient.

Recently \cite{FJ7}, as an alternative to the Perron vector of $A,$ we have
proposed the (right) Perron vector of $AA^{T}$ (called the singular vector of
$A)$ as the weight vector obtained from the reciprocal matrix $A.$ Since both
vectors are in $\mathcal{C}(A),$ if $\mathcal{C}(A)\subseteq\mathcal{E}(A),$
then they are efficient for $A.$

From the known description of the efficient vectors for a $2$-block perturbed
consistent matrix \cite{CFF}, it follows that, for such a matrix $A$,
$\mathcal{E}(A)$ is convex and, thus, $\mathcal{C}(A)\subseteq\mathcal{E}(A).$
Moreover, if $A$ is a $3$-by-$3$ matrix, then $\mathcal{C}(A)=\mathcal{E}(A)$
(Proposition \ref{prop3}).

We call a reciprocal matrix obtained from a consistent matrix by modifying at
most $3$ pairs of reciprocal entries a triple perturbed consistent matrix.
Such matrices are $s$-block perturbed consistent matrices for some $3\leq
s\leq6,$ depending on the pattern of the perturbation.

In this paper we focus on triple perturbed consistent matrices with the
modified entries located in a $4$-by-$4$ principal submatrix. When all the
modified entries above (below) the diagonal lie in the same row (or column),
it is known that $\mathcal{C}(A)\subseteq\mathcal{E}(A)$ \cite{FJ7}.
Otherwise, we have two possible situations: the set of rows and the set of
columns in which the modified entries above the diagonal lie have a nonempty
intersection, in which case the matrix is a $3$-block perturbed consistent
matrix; the modified entries above the diagonal lie in a $2$-by-$2$ submatrix
indexed by disjoint sets of rows and columns. In the latter case, we say that
the matrix is a $4$-block triangular perturbed consistent matrix. In both
cases, sufficient conditions for the efficiency of the Perron vector are known
(see \cite{FJ4} for the case of a $3$-block perturbed consistent matrix and
\cite{FerFur,Ro} for a $4$-block triangular perturbed consistent matrix). Here
we show that, in fact, these conditions (including some boundary ones that are
new) are necessary and sufficient for the efficiency of all vectors in the
convex hull of the columns of such reciprocal matrices. Thus, when these
conditions hold, the efficiency of the Perron vector and the singular vector
is a corollary of our results and a much larger class of efficient vectors is provided.

\bigskip

The paper is organized as follows. In Section \ref{s22} we introduce the
technical background that we need, including some additional notation. In
Section \ref{s3} we identify a necessary condition for $\mathcal{C}(A)$ to be
contained in $\mathcal{E}(A)$ for a general reciprocal matrix $A.$ In Section
\ref{s4} we give necessary and sufficient conditions for $\mathcal{C}(A)$ to
be contained in $\mathcal{E}(A)$ when $A$ is a $3$-block perturbed consistent
matrix. This includes the case of a $4$-by-$4$ reciprocal matrix. In Section
\ref{s6} we solve the same problem when $A$ is a $4$-block triangular
perturbed consistent matrix. In Section \ref{s7}, for a $3$-block perturbed
consistent matrix $A$ for which $\mathcal{C}(A)$ is contained in
$\mathcal{E}(A),$ we compare numerically the performance of different convex
combinations of the columns of $A$ and weighted geometric means of the
columns. In particular, we consider the Perron vector, the singular vector and
the arithmetic and geometric means of all columns in our comparisons. Our
results strongly emphasize the importance of having a large class of efficient
vectors from which a good weight vector can be chosen, as the performance of
these vectors seems to depend on the particular reciprocal matrix. We conclude
the paper with some final remarks in Section \ref{s8}.

\section{Background\label{s22}}

In this section we introduce some additional notation and some known facts
that will be used in obtaining our main results.

\subsection{Notation\label{secnotation}}

We denote by $\mathcal{V}_{n}$ the set of positive column $n$-vectors. Given
$w\in\mathcal{V}_{n},$ we usually denote by $w_{i}$ the $i$th entry of $w.$

We denote by $J_{m,n}$ the $m$-by-$n$ matrix with all entries equal to $1$ and
by $\mathbf{e}_{n}$ the vector in $\mathcal{V}_{n}$ with all entries equal to
$1.$ By $I_{n}$ we denote the identity matrix of size $n.$ Given a vector $w,$
we denote by $\operatorname*{diag}(w)$ the diagonal matrix with main diagonal
$w.$ If $w$ is positive, we say that $\operatorname*{diag}(w)$ is a positive
diagonal matrix.

For an $n$-by-$n$ matrix $A=[a_{ij}],$ the principal submatrix of $A$
determined by deleting (by retaining) the rows and columns indexed by a subset
$K\subseteq\{1,\ldots,n\}$ is denoted by $A(K)$ $(A[K]);$ we abbreviate
$A(\{i\})$ as $A(i)$ and $A(\{i,j\})$ as $A(i,j)$, $i\neq j.$ Similarly, if
$w\in\mathcal{V}_{n},$ we denote by $w(K)$ ($w[K]$) the vector obtained from
$w$ by deleting (by retaining) the entries indexed by $K$ and abbreviate
$w(\{i\})$ as $w(i)$ and $w(\{i,j\})$ as $w(i,j)$. Note that, if $A$ is
reciprocal (consistent) then so are $A(K)$ and $A[K].$

\subsection{Basic results on efficient vectors}

An important characterization of efficient vectors for a reciprocal matrix in
terms of a certain directed graph is known.

Given $A=[a_{ij}]\in\mathcal{PC}_{n}$ and a vector
\begin{equation}
w=\left[
\begin{array}
[c]{ccc}%
w_{1} & \cdots & w_{n}%
\end{array}
\right]  ^{T}\in\mathcal{V}_{n}, \label{ww1}%
\end{equation}
define $G(A,w)$ as the directed graph with vertex set $\{1,\ldots,n\}$ and a
directed edge $i\rightarrow j$ if and only if $w_{i}\geq a_{ij}w_{j}$, $i\neq
j.$

The following result was noticed in \cite{blanq2006} using techniques from
optimization. In \cite{FJ3} a matricial proof of it was given.

\begin{theorem}
\label{blanq} Let $A\in\mathcal{PC}_{n}$ and $w\in\mathcal{V}_{n}$. The vector
$w$ is efficient for $A$ if and only if $G(A,w)$ is a strongly connected
digraph, that is, for all pairs of vertices $i,j,$ with $i\neq j,$ there is a
directed path from $i$ to $j$ in $G(A,w)$.
\end{theorem}

\bigskip

The set $\mathcal{PC}_{n}$ is closed under similarity by either a permutation
matrix or a positive diagonal matrix (and, thus, a monomial similarity). Such
transformations interface with efficient vectors nicely.

\begin{lemma}
\cite{Fu22}\label{lsim} Suppose that $A\in\mathcal{PC}_{n}$ and $w\in
\mathcal{E}(A).$ If $S$ is an $n$-by-$n$ monomial matrix, then $SAS^{-1}%
\in\mathcal{PC}_{n}$ and $Sw\in\mathcal{E}(SAS^{-1})$.
\end{lemma}

In \cite{FJ2} we gave an inductive construction of all the efficient vectors
for a reciprocal matrix $A\in\mathcal{PC}_{n}$. They can be constructed as
extensions of the efficient vectors for the $3$-by-$3$ principal submatrices
of $A$. This procedure is based on the "only if" claim in the next theorem.
The "if" statement was noticed in \cite{FJ5}. This theorem will play a central
role in obtaining our main results.

\begin{theorem}
\label{tnbyn}Let $A\in\mathcal{PC}_{n}$, $n\geq4,$ and $w\in\mathcal{V}_{n}.$
Then, $w\in\mathcal{E}(A)$ if and only if there are $i,j\in\{1,\ldots,n\},$
with $i\neq j,$ such that $w(i)\in\mathcal{E}(A(i))$ and $w(j)\in
\mathcal{E}(A(j)).$
\end{theorem}

We note that the "if" statement in the theorem also holds for $n=3.$

\bigskip

Let $a_{1},\ldots,a_{s}\in\mathcal{V}_{n}$ and $\alpha_{1},\ldots,\alpha_{s}$
be nonnegative numbers summing to one, and denote by $a_{i}^{(\alpha_{i})}$
the entry-wise power of $a_{i}$ with exponent $\alpha_{i}.$ We call the
Hadamard product
\begin{equation}
a_{1}^{(\alpha_{1})}\circ a_{2}^{(\alpha_{2})}\circ\cdots\circ a_{s}^{\left(
\alpha_{s}\right)  } \label{geomean}%
\end{equation}
the $(\alpha_{1},\ldots,\alpha_{s})$-weighted geometric mean of the columns
$a_{1},\ldots,a_{s}.$ When the specific values of $\alpha_{1},\ldots
,\alpha_{s}$ are unidentified, we just say "a weighted geometric mean of the
columns $a_{1},\ldots,a_{s}$"$.$ If $\alpha_{1}=\cdots=\alpha_{s}=\frac{1}%
{s},$ we simply call (\ref{geomean}) the geometric mean of the columns
$a_{1},\ldots,a_{s}.$ In \cite{FJ3} it was shown that any weighted geometric
mean of the columns of $A\in\mathcal{PC}_{n}$ is efficient for $A.$

\subsection{The convex hull of the columns of a reciprocal matrix}

Denote by $\mathcal{C}(A)$ the (convex) cone generated by the columns of
$A\in\mathcal{PC}_{n}$, that is, the set of nonzero linear combinations, with
nonnegative coefficients, of the columns of $A$. Note that, according to our
definition, $0$ is not in $\mathcal{C}(A)$. The right Perron vector of $A$ and
the Perron vector of $AA^{T}$ lie in $\mathcal{C}(A)$ \cite{FJ7}$.$

We recall that the (right) Perron vector of an $n$-by-$n$ positive matrix $A$
is a (right) eigenvector of $A$ associated with the spectral radius $\rho(A)$
of $A$ (called the Perron eigenvalue of $A$) \cite{HJ}. It is known that, up
to a constant factor, all its entries are positive. And, after normalization
(such as making the entry sum $1$) it is unique.

When $A$ is consistent, $A=\left[  \frac{w_{i}}{w_{j}}\right]  $ for some
$w\in\mathcal{V}_{n},$ any column of $A\ $ is a multiple of $w.$ In this case,
$\mathcal{C}(A)$ is formed by the positive multiples of $w$ and is precisely
$\mathcal{E}(A).$

Any column of $A\in\mathcal{PC}_{n}$ is efficient for $A$ \cite{FJ1}. However,
it may happen that $\mathcal{C}(A)$ is not contained in $\mathcal{E}(A)$,
contrasting with the fact that the set of all vectors in the geometric convex
hull of the columns of $A$ is contained in $\mathcal{E}(A)$ \cite{FJ3}. Of
course, in this case $\mathcal{E}(A)$ is not convex. When $\mathcal{C}(A)$ is
contained in $\mathcal{E}(A),$ the efficiency of both the Perron vector and
the singular vector of $A$ is guaranteed (though each or both of these vectors
may be efficient when $\mathcal{C}(A)$ is not contained in $\mathcal{E}(A)$
\cite{FJ7}).

\bigskip

The following lemma was noted in \cite{FJ7}.

\begin{lemma}
\label{lconvex2}Let $A$ and $S$ be $n$-by-$n$ matrices, with $S$ monomial. If
$\mathcal{C}(A)$ is the cone generated by the columns of $A$ then
$S\mathcal{C}(A)$ is the cone generated by the columns of $SAS^{-1}.$
\end{lemma}

From Lemmas \ref{lsim} and \ref{lconvex2}, we get the following.

\begin{lemma}
\label{cconv}Let $A\in\mathcal{PC}_{n}$ and $S$ be an $n$-by-$n$ monomial
matrix. Then, $\mathcal{C}(A)\subseteq\mathcal{E}(A)$ if and only if
$\mathcal{C}(SAS^{-1})\subseteq\mathcal{E}(SAS^{-1})$.
\end{lemma}

\subsection{Efficient vectors for block perturbed consistent matrices}

We say that a reciprocal matrix $M\in\mathcal{PC}_{n}$ is an $s$\emph{-block
perturbed consistent matrix}$,$ $1<s<n,$ if it is obtained from a consistent
matrix by modifying (reciprocally) an $s$-by-$s$ principal submatrix. Such a
matrix is monomially similar to a matrix of the form%
\begin{equation}
A=A_{n}(B):=\left[
\begin{array}
[c]{cc}%
B & J_{s,n-s}\\
J_{n-s,s} & J_{n-s}%
\end{array}
\right]  \in\mathcal{PC}_{n}, \label{BJ}%
\end{equation}
with $B\in\mathcal{PC}_{s}$ \cite{FJ4}$.$ In fact, if $M$ is obtained from the
consistent matrix $\left[  \frac{w_{i}}{w_{j}}\right]  ,$ then, for
$D=\operatorname*{diag}(w),$ $D^{-1}MD$ has all entries equal to $1,$ except
those in the positions of the modified entries. Then, by a permutation
similarity, the $s$-by-$s$ principal submatrix containing the modified entries
may be put in the left upper corner.

Based on Lemmas \ref{lsim} and \ref{cconv}, for the purpose of studying the
efficiency of the vectors in the convex hull of the columns of an $s$-block
perturbed consistent matrix, we may assume the matrix has form (\ref{BJ}).

In \cite{CFF} the description of the efficient vectors for a $2$-block
perturbed consistent matrix was given. Such a matrix is monomially similar to
a matrix of the form%
\[
S_{n}(x)=\left[
\begin{tabular}
[c]{c|c}%
$%
\begin{array}
[c]{cc}%
1 & x\\
\frac{1}{x} & 1
\end{array}
$ & $J_{2,n-2}$\\\hline
$J_{n-2,2}$ & $J_{n-2}$%
\end{tabular}
\ \ \ \right]  \in\mathcal{PC}_{n},
\]
for some $x>0.$

\begin{theorem}
\cite{CFF}\label{tmain} Let $n\geq3,$ $x>0$. Then, a vector $w\in
\mathcal{V}_{n}$ as in (\ref{ww1}) is efficient for $S(x)$ if and only if%
\[
w_{2}\leq w_{k}\leq w_{1}\leq xw_{2}\qquad\text{for }3\leq k\leq n
\]
or%
\[
w_{2}\geq w_{k}\geq w_{1}\geq xw_{2}\qquad\text{for }3\leq k\leq n.\text{ }%
\]

\end{theorem}

Any matrix $A\in\mathcal{PC}_{3}$ is a $2$-block perturbed consistent matrix
\cite{FJ2}. Thus, all the efficient vectors for $A\in\mathcal{PC}_{3}$ are
obtained from Theorem \ref{tmain} using Lemma \ref{lsim}. We explicitly state
this result since it will be used often.

\begin{corollary}
\cite{FJ2}\label{c3por3} Let
\[
\left[
\begin{array}
[c]{ccc}%
1 & a_{12} & a_{13}\\
\frac{1}{a_{12}} & 1 & a_{23}\\
\frac{1}{a_{13}} & \frac{1}{a_{23}} & 1
\end{array}
\right]  .
\]
Then, $w=\left[
\begin{array}
[c]{ccc}%
w_{1} & w_{2} & w_{3}%
\end{array}
\right]  ^{T}\in\mathcal{V}_{3}$ is efficient for $A$ if and only if%
\[
a_{23}w_{3}\leq w_{2}\leq\frac{w_{1}}{a_{12}}\leq\frac{a_{13}}{a_{12}}%
w_{3}\text{\qquad or \qquad}a_{23}w_{3}\geq w_{2}\geq\frac{w_{1}}{a_{12}}%
\geq\frac{a_{13}}{a_{12}}w_{3}.
\]

\end{corollary}

\bigskip

Next we characterize the triple perturbed consistent matrices for which the
entries modified from a consistent matrix lie in a $4$-by-$4$ principal
submatrix, which we call $4$\emph{-block triple perturbed consistent matrices}.

A $4$-block triple perturbed consistent matrix in $\mathcal{PC}_{n}$,
$n\geq5,$ monomial similar to a matrix of the form
\begin{equation}
\left[
\begin{tabular}
[c]{c|c}%
$%
\begin{array}
[c]{cccc}%
1 & 1 & a_{13} & a_{14}\\
1 & 1 & 1 & a_{24}\\
\frac{1}{a_{13}} & 1 & 1 & 1\\
\frac{1}{a_{14}} & \frac{1}{a_{24}} & 1 & 1
\end{array}
$ & $J_{4,n-4}$\\\hline
$J_{n-4,4}$ & $J_{n-4}$%
\end{tabular}
\ \right]  \label{triple}%
\end{equation}
is called a $4$\emph{-block triangular perturbed consistent matrix}. A
reciprocal matrix obtained from a consistent one by modifying one row (and
column) is called a \emph{column perturbed consistent matrix}.

As mentioned in the introduction, based on possible intersections of the row
indices and the column indices of the perturbed entries in a $4$-block triple
perturbed consistent matrix, we have the following.

\begin{proposition}
\label{proptriple}Let $A\in\mathcal{PC}_{n}$ be a $4$-block triple perturbed
consistent matrix. Then $A$ is either a column perturbed consistent matrix, a
$3$-block perturbed consistent matrix or a $4$-block triangular perturbed
consistent matrix.
\end{proposition}

In \cite{FJ7} it was shown that, if $A\in\mathcal{PC}_{n}$ is a column
perturbed consistent matrix, then $\mathcal{C}(A)\subseteq\mathcal{E}(A).$ In
particular, the singular vector and the Perron vector of such a matrix are efficient.

Here we show that the known sufficient conditions for the efficiency of the
Perron vector of a matrix $A$ that is either a $3$-block perturbed consistent
matrix \cite{FJ4} or a $4$-block triangular perturbed consistent matrix
\cite{FerFur,Ro} (with some additional boundary conditions in the latter case)
are, in fact, necessary and sufficient for all the vectors in $\mathcal{C}(A)$
to be efficient. Thus, taking into account Proposition \ref{proptriple} and
the result in \cite{FJ7} for column perturbed consistent matrices, a complete
characterization of the $4$-block triple perturbed consistent matrices $A$ for
which $\mathcal{C}(A)\subseteq\mathcal{E}(A)$ follows.

\bigskip

As a consequence of the next known lemma, we give a helpful result in the
study of the efficiency of the convex combinations of the columns of an
$s$-block perturbed consistent matrix.

\begin{lemma}
\cite{FJ4}\label{laux} Let $A\in\mathcal{PC}_{n}$ as in (\ref{BJ}), with
$B\in\mathcal{PC}_{s}$ and $n>s+1.$ Let $w=\left[
\begin{array}
[c]{cccc}%
w_{1} & w_{2} & \cdots & w_{n}%
\end{array}
\right]  ^{T}\in\mathcal{V}_{n}.$ If $w_{p}=w_{q},$ for some $p,q\in
\{s+1,\ldots,n\}$ with $p\neq q,$ then $w$ is efficient for $A$ if and only if
$w(p)$ is efficient for $A(p)=A_{n-1}(B).$
\end{lemma}

We then have the following.

\begin{lemma}
\label{lemablock}Let $A\in\mathcal{PC}_{n}$ as in (\ref{BJ}), with
$B\in\mathcal{PC}_{s}$ and $n>s+1.$ Then $\mathcal{C}(A)\subseteq
\mathcal{E}(A)$ if and only if $\mathcal{C}(A\left[  \{1,\ldots,s+1\}\right]
)\subseteq\mathcal{E}(A\left[  \{1,\ldots,s+1\}\right]  ).$
\end{lemma}

\begin{proof}
($\Leftarrow$) Let $w\in\mathcal{C}(A).$ Then $w$ has the form
\begin{equation}
w=\left[
\begin{array}
[c]{cccccc}%
w_{1} & \cdots & w_{s} & w_{s+1} & \cdots & w_{s+1}%
\end{array}
\right]  ^{T}\in\mathcal{V}_{n}. \label{ws}%
\end{equation}
Let%
\begin{equation}
u=\left[
\begin{array}
[c]{cccc}%
w_{1} & \cdots & w_{s} & w_{s+1}%
\end{array}
\right]  ^{T}. \label{uu}%
\end{equation}
We have $u\in\mathcal{C}(A\left[  \{1,\ldots,s+1\}\right]  ).$ By the
hypothesis, $u\in\mathcal{E}(A\left[  \{1,\ldots,s+1\}\right]  ).$ By Lemma
\ref{laux}, $w\in\mathcal{E}(A).$

($\Rightarrow$) Let $u\in\mathcal{C}(A\left[  \{1,\ldots,s+1\}\right]  ),$
with $u$ as in (\ref{uu}). Then $w\in\mathcal{C}(A),$ with $w$ as in
(\ref{ws}). By the hypothesis, $w\in\mathcal{E}(A).$ By Lemma \ref{laux},
$u\in\mathcal{E}(A\left[  \{1,\ldots,s+1\}\right]  ).$
\end{proof}

\section{A necessary condition for the efficiency of $\mathcal{C}%
(A)$\label{s3}}

When $A\in\mathcal{PC}_{n}$ is a $2$-block perturbed consistent matrix, the
efficient vectors for $A$ are defined by a finite set of linear inequalities
in their entries, as follows from Lemma \ref{lsim} and Theorem \ref{tmain}.
Thus, $\mathcal{E}(A)$ is convex. Since any column of $A$ is efficient for
$A,$ we get that $\mathcal{C}(A)\subseteq\mathcal{E}(A).$ When $n=3,$ we have
the following.

\begin{proposition}
\label{prop3}Let $A\in\mathcal{PC}_{3}$. Then $\mathcal{E}(A)=\mathcal{C}(A)$.
\end{proposition}

\begin{proof}
Since $A$ is a $2$-block perturbed consistent matrix, we have $\mathcal{C}%
(A)\subseteq\mathcal{E}(A).$ Thus, we just need to see that $\mathcal{E}%
(A)\subseteq\mathcal{C}(A).$ By Lemmas \ref{lsim} and \ref{lconvex2}, we may
assume that
\[
A=\left[
\begin{array}
[c]{ccc}%
1 & 1 & x\\
1 & 1 & 1\\
\frac{1}{x} & 1 & 1
\end{array}
\right]  ,
\]
with $x\geq1.$ Suppose that $w$ is efficient for $A.$ Then, by Theorem
\ref{tmain} (or Corollary \ref{c3por3}),
\begin{equation}
w_{3}\leq w_{2}\leq w_{1}\leq xw_{3}. \label{ww}%
\end{equation}
If all entries of $w$ are equal then $w\in\mathcal{C}(A)$, since any positive
vector proportional to a column is efficient. Suppose that $w$ has two
distinct entries, in which case $x>1.$ We have $w=Ay$ for%
\[
y=\left[
\begin{array}
[c]{c}%
\frac{w_{2}-w_{3}}{x-1}x\\
\frac{xw_{3}-w_{1}}{x-1}\\
\frac{w_{1}-w_{2}}{x-1}%
\end{array}
\right]  .
\]
By (\ref{ww}), $y$ is nonzero and nonnegative. Thus, $w\in\mathcal{C}(A).$
\end{proof}

\bigskip

Next we give a sufficient condition for $\mathcal{C}(A)$ to not be contained
in $\mathcal{E}(A)$, when $A\in\mathcal{PC}_{n}$, $n\geq4$. Taking into
account Lemmas \ref{lsim} and \ref{cconv}, we assume that all entries in the
last row and column of $A$ are $1.$ In fact, if $D$ is the diagonal matrix
whose diagonal is the $i$th column of $A^{\prime},$ then the $i$th row and
column of $D^{-1}A^{\prime}D$ have all entries equal to $1.$ With an
additional permutation similarity, the row and column can be made the last.
Thus, given a matrix in $\mathcal{PC}_{n},$ the next theorem can be applied to
each matrix monomial similar to it and with the last row and column having all
entries equal to $1.$

\begin{theorem}
Let $B\in\mathcal{PC}_{n-1},$ $n\geq4,$ and
\[
A=\left[
\begin{array}
[c]{cc}%
B & \mathbf{e}_{n-1}\\
\mathbf{e}_{n-1}^{T} & 1
\end{array}
\right]  \in\mathcal{PC}_{n}.
\]
Let $y=\left[
\begin{array}
[c]{ccc}%
y_{1} & \cdots & y_{n-1}%
\end{array}
\right]  ^{T}$ be a nonnegative vector with entries summing to $1.$ If $By$ is
(entry-wise) strictly greater than $\mathbf{e}_{n-1}$ or strictly less than
$\mathbf{e}_{n-1},$ then $A\left[
\begin{array}
[c]{cc}%
y^{T} & 0
\end{array}
\right]  ^{T}\notin\mathcal{E}(A).$ Thus, $\mathcal{C}(A)$ is not contained in
$\mathcal{E}(A).$ In particular, $\mathcal{E}(A)$ is not convex.
\end{theorem}

\begin{proof}
Let $A=[a_{ij}]\ $and let
\[
w=A\left[
\begin{array}
[c]{cccc}%
y_{1} & \cdots & y_{n-1} & 0
\end{array}
\right]  ^{T}=\left[
\begin{array}
[c]{c}%
By\\
1
\end{array}
\right]  .
\]
Note that $w\in\mathcal{C}(A).$ If $By$ is strictly greater than
$\mathbf{e}_{n-1}$, then, for all $i=1,\ldots,n-1,$ we have $\frac{w_{i}%
}{w_{n}}=w_{i}>1=a_{in}$, in which case $n$ is a sink vertex of $G(A,w)$. If
$By$ is strictly less than $\mathbf{e}_{n-1}$, then, for all $i=1,\ldots,n-1,$
we have $\frac{w_{i}}{w_{n}}=w_{i}<1=a_{in}$, in which case $n$ is a source
vertex of $G(A,w)$. In any case, $G(A,w)$ is not strongly connected, implying
that $w$ is inefficient for $A,$ by Theorem \ref{blanq}.
\end{proof}

\bigskip

We observe that in the theorem the number of nonzero entries in $y$ should be
at least $3.$

\bigskip

\begin{example}
Let%
\[
A=\left[
\begin{array}
[c]{cccc}%
1 & 4 & \frac{1}{6} & 1\\
\frac{1}{4} & 1 & 5 & 1\\
6 & \frac{1}{5} & 1 & 1\\
1 & 1 & 1 & 1
\end{array}
\right]  .
\]
The row sum vector of the $3$-by-$3$ principal submatrix $A(4)$ of $A$ is
\[
A(4)\mathbf{e}_{3}=\left[
\begin{array}
[c]{c}%
1+4+\frac{1}{6}\\
\frac{1}{4}+1+5\\
6+\frac{1}{5}+1
\end{array}
\right]  =\left[
\begin{array}
[c]{c}%
\frac{31}{6}\\
\frac{25}{4}\\
\frac{36}{5}%
\end{array}
\right]  ,
\]
which has all entries strictly greater than $3$ (that is, $A(4)(\frac{1}%
{3}\mathbf{e}_{3})$ has all entries strictly greater than $1$)$.$ Then
\[
A\left[
\begin{array}
[c]{c}%
1\\
1\\
1\\
0
\end{array}
\right]  =\left[
\begin{array}
[c]{c}%
\frac{31}{6}\\
\frac{25}{4}\\
\frac{36}{5}\\
3
\end{array}
\right]
\]
is inefficient for $A,$ implying that $\mathcal{C}(A)$ is not contained in
$\mathcal{E}(A)$.
\end{example}

\section{Efficiency of the convex hull of the columns of a $3$-block perturbed
consistent matrix\label{s4}}

Here we give necessary and sufficient conditions for $\mathcal{C}%
(A)\subseteq\mathcal{E}(A)$ when $A$ is a $3$-block perturbed consistent
matrix. By Lemmas \ref{lsim} and \ref{cconv}, we may assume that $A$ is as in
(\ref{BJ}) with $B\in\mathcal{PC}_{3}.$ It is easy to see that there is a
$3$-by-$3$ permutation matrix $P$ such that, for $A^{\prime}=(P\oplus
I_{n-3})A(P^{T}\oplus I_{n-3}),\ $the principal submatrix $A^{\prime
}[\{1,2,3\}]$ of $A^{\prime}$ has one of the following forms:
\[
\text{ }\left[
\begin{array}
[c]{ccc}%
1 & \geq1 & \geq1\\
\leq1 & 1 & \geq1\\
\leq1 & \leq1 & 1
\end{array}
\right]  \text{ \qquad or \qquad}\left[
\begin{array}
[c]{ccc}%
1 & <1 & >1\\
>1 & 1 & <1\\
\leq1 & >1 & 1
\end{array}
\right]  ,
\]
in which $\geq1$ (resp. $\leq1,$ $<1,$ $>1$) represents an entry $\geq1$
(resp. $\leq1,$ $<1,$ $>1$). Moreover, with a possible additional permutation
similarity of the same type, we may, and do, assume that $A$ has the form
\begin{equation}
\left[
\begin{tabular}
[c]{c|c}%
$%
\begin{array}
[c]{ccc}%
1 & a_{12} & a_{13}\\
\frac{1}{a_{12}} & 1 & a_{23}\\
\frac{1}{a_{13}} & \frac{1}{a_{23}} & 1
\end{array}
$ & $J_{3,n-3}$\\\hline
$J_{n-3,3}$ & $J_{n-3,n-3}$%
\end{tabular}
\ \right]  , \label{MAtriple}%
\end{equation}
with $a_{12},a_{13},a_{23}>0$ satisfying one of the following four conditions:%
\begin{align}
\text{i) }a_{13}  &  \geq1\text{ and }a_{12},a_{23}>1,\text{ \quad ii) }%
a_{12}\geq a_{13}\geq a_{23}=1,\text{ }\nonumber\\
& \label{cond11}\\
\text{iii) }a_{23}  &  \geq a_{13}\geq a_{12}=1,\text{ \quad or\quad\ iv)
}a_{13}>1\text{ and }a_{12},a_{23}<1.\nonumber
\end{align}

We start by considering the case $A\in\mathcal{PC}_{4}:$%
\begin{equation}
A=\left[
\begin{array}
[c]{cccc}%
1 & a_{12} & a_{13} & 1\\
\frac{1}{a_{12}} & 1 & a_{23} & 1\\
\frac{1}{a_{13}} & \frac{1}{a_{23}} & 1 & 1\\
1 & 1 & 1 & 1
\end{array}
\right]  . \label{MA}%
\end{equation}

\begin{remark}
\label{Reffsub}Let $A\in\mathcal{PC}_{4}$ be as in (\ref{MA}) and let
$w=\left[
\begin{array}
[c]{cccc}%
w_{1} & w_{2} & w_{3} & w_{4}%
\end{array}
\right]  ^{T}\in\mathcal{V}_{4}.$ According to Corollary \ref{c3por3}, we have:
\end{remark}

\begin{itemize}
\item $w(1)$ is efficient for $A(1)$ if and only if
\[
w_{3}\leq w_{4}\leq w_{2}\leq a_{23}w_{3}\text{ \quad or\quad}w_{3}\geq
w_{4}\geq w_{2}\geq a_{23}w_{3}\text{; }%
\]

\item $w(2)$ is efficient for $A(2)$ if and only if
\[
w_{3}\leq w_{4}\leq w_{1}\leq a_{13}w_{3}\text{ \quad or\quad}w_{3}\geq
w_{4}\geq w_{1}\geq a_{13}w_{3}\text{;}%
\]

\item $w(3)$ is efficient for $A(3)$ if and only if
\[
w_{2}\leq w_{4}\leq w_{1}\leq a_{12}w_{2}\text{ \quad or\quad}w_{2}\geq
w_{4}\geq w_{1}\geq a_{12}w_{2}\text{;}%
\]

\item $w(4)$ is efficient for $A(4)$ if and only if
\[
a_{13}w_{3}\leq w_{1}\leq a_{12}w_{2}\leq a_{23}a_{12}w_{3}\text{ \quad
or\quad}a_{13}w_{3}\geq w_{1}\geq a_{12}w_{2}\geq a_{23}a_{12}w_{3}.
\]

\end{itemize}

Next we give sufficient conditions on the entries of $A\ $for all vectors in
$\mathcal{C}(A)$ to be efficient.

\begin{lemma}
\label{lsuff}Let $A\in\mathcal{PC}_{4}$ be as in (\ref{MA}), with
$a_{12},a_{13},a_{23}$ satisfying (\ref{cond11}). If $a_{13}\leq a_{12}a_{23}%
$, then $\mathcal{C}(A)\subseteq\mathcal{E}(A).$
\end{lemma}

\begin{proof}
Suppose that (\ref{cond11}) holds and $a_{13}\leq a_{12}a_{23}.$ Then,
$a_{12},a_{23}\geq1.$ Let $w\in\mathcal{C}(A),$ that is, $w=Au$ for some
nonnegative vector $u=\left[
\begin{array}
[c]{cccc}%
u_{1} & u_{2} & u_{3} & u_{4}%
\end{array}
\right]  ^{T}$ such that $\sum_{i=1}^{4}u_{i}=1.$ Then

\medskip%

\begin{tabular}
[c]{ll}%
$w_{3}-w_{4}=u_{1}\left(  \frac{1}{a_{13}}-1\right)  +u_{2}\left(  \frac
{1}{a_{23}}-1\right)  \leq0;\text{ }$ & \\
& \\
$w_{4}-w_{1}=u_{2}\left(  1-a_{12}\right)  +u_{3}\left(  1-a_{13}\right)
\leq0;$ & \\
& \\
$w_{1}-a_{12}w_{2}=u_{3}\left(  a_{13}-a_{12}a_{23}\right)  +u_{4}\left(
1-a_{12}\right)  \leq0;$ & \\
& \\
$w_{2}-a_{23}w_{3}=u_{1}\left(  \frac{1}{a_{12}}-\frac{a_{23}}{a_{13}}\right)
+u_{4}\left(  1-a_{23}\right)  \leq0.$ &
\end{tabular}

\bigskip

Taking into account Remark \ref{Reffsub}, we may conclude the following.

\begin{itemize}
\item If $w_{4}\leq w_{2}$ then $w(1)$ is efficient for $A(1),$ otherwise
$w(3)$ is efficient for $A(3).$

\item If $w_{1}\leq a_{13}w_{3}$ then $w(2)$ is efficient for $A(2),$
otherwise $w(4)$ is efficient for $A(4).$
\end{itemize}

In any case, there are $i,j\in\{1,2,3,4\},$ with $i\neq j,$ such that
$w(i)\in\mathcal{E}(A(i))$ and $w(j)\in\mathcal{E}(A(j))$. Thus, by Theorem
\ref{tnbyn}, $w$ is efficient for $A.$
\end{proof}

\bigskip

Next we show that, when $A\in\mathcal{PC}_{4}$ is as in (\ref{MA}), with
$a_{12},a_{13},a_{23}$ satisfying (\ref{cond11}), the sufficient conditions in
Lemma \ref{lsuff} for $\mathcal{C}(A)$ to be contained in $\mathcal{E}(A)$ are
also necessary.

\begin{lemma}
\label{lnec}Let $A\in\mathcal{PC}_{4}$ be as in (\ref{MA}), with
$a_{12},a_{13},a_{23}$ satisfying (\ref{cond11})$.$ If $a_{13}>a_{12}a_{23},$
then $\mathcal{C}(A)$ is not contained in $\mathcal{E}(A).$
\end{lemma}

\begin{proof}
If $a_{12},a_{13},a_{23}$ satisfy (\ref{cond11}) and $a_{13}>a_{12}a_{23}$,
then one of the following conditions holds:

\begin{enumerate}
\item[i)] $a_{12},a_{23}>1$ and $a_{13}>a_{12}a_{23},$ or

\item[ii)] $a_{12},a_{23}<1$ and $a_{13}>1.$
\end{enumerate}

In each of the two cases, we give a vector in $\mathcal{C}(A)$ that is not in
$\mathcal{E}(A)$.

\bigskip

\textbf{Case i)} Suppose that $a_{12},a_{23}>1$ and $a_{13}>a_{12}a_{23}.$ Let
$\varepsilon>0$ and let $\ $
\begin{align*}
u_{4}  &  =u_{2}=1;\\
u_{1}  &  =\frac{a_{12}(a_{23}-1)}{a_{13}-a_{23}a_{12}}+\varepsilon;\\
u_{3}  &  =\frac{a_{12}-1}{a_{13}-a_{12}a_{23}}+\frac{a_{12}-1}{2a_{12}%
(a_{23}-1)}\varepsilon.
\end{align*}
Let $w=Au,$ with $u=\left[
\begin{array}
[c]{cccc}%
u_{1} & u_{2} & u_{3} & u_{4}%
\end{array}
\right]  ^{T}$. As $u$ is a nonzero nonnegative vector, $w\in\mathcal{C}(A).$
Since
\[
w_{4}-w_{2}=\frac{a_{12}-1}{2a_{12}}\varepsilon>0,
\]
by Remark \ref{Reffsub} we have that $w(1)$ is inefficient for $A(1);$ since
\[
w_{1}-a_{12}w_{2}=\frac{\left(  a_{13}-a_{12}a_{23}\right)  (a_{12}%
-1)}{2a_{12}(a_{23}-1)}\varepsilon>0,
\]
we have that $w(3)$ is inefficient for $A(3);$ also, for $\varepsilon$
sufficiently small,
\[
w_{2}-a_{23}w_{3}=\frac{(a_{23}-1)(1-a_{13})}{a_{13}}+\frac{a_{13}%
-a_{12}a_{23}}{a_{12}a_{13}}\varepsilon<0,
\]
implying that $w(4)$ is inefficient for $A(4).$ By Theorem \ref{tnbyn}, for
$\varepsilon$ sufficiently small, $w$ is inefficient for $A.$

\bigskip

\textbf{Case ii)} Suppose that $a_{12},a_{23}<1$ and $a_{13}>1.\ $Let
$\varepsilon>0$ and
\begin{align*}
u_{4}  &  =u_{3}=1;\\
u_{1}  &  =\frac{a_{13}(a_{23}-1)}{a_{23}(a_{12}-1)}-\varepsilon;\\
u_{2}  &  =\frac{a_{13}-1}{1-a_{12}}-\frac{a_{23}(a_{13}-1)}{2a_{13}%
(1-a_{23})}\varepsilon.
\end{align*}
Let $w=Au,$ with $u=\left[
\begin{array}
[c]{cccc}%
u_{1} & u_{2} & u_{3} & u_{4}%
\end{array}
\right]  ^{T}$. For $\varepsilon$ sufficiently small, $u$ is a nonzero
nonnegative vector and, thus, $w\in\mathcal{C}(A).$ Moreover,
\[
w_{4}-w_{2}=\frac{\left(  a_{23}-1\right)  \left(  a_{13}-a_{12}a_{23}\right)
}{a_{12}a_{23}}+\frac{1-a_{12}}{a_{12}}\varepsilon<0,
\]
implying, by Remark \ref{Reffsub}, that $w(1)$ is inefficient for $A(1).$ We
also have
\[
w_{3}-w_{4}=\frac{a_{13}-1}{2a_{13}}\varepsilon>0,
\]
implying that $w(2)$ is inefficient for $A(2);$ and
\[
w_{4}-w_{1}=\frac{a_{23}\left(  1-a_{12}\right)  (a_{13}-1)}{2a_{13}%
(a_{23}-1)}\varepsilon<0,
\]
implying that $w(3)$ is inefficient for $A(3).$ By Theorem \ref{tnbyn}, for
$\varepsilon$ sufficiently small, $w$ is inefficient for $A.$
\end{proof}

\bigskip

\bigskip We then present the main results of this section. From Lemmas
\ref{lsuff} and \ref{lnec} we have the characterization of the matrices
$A\in\mathcal{PC}_{4}$ as in (\ref{MA}) for which $\mathcal{C}(A)\subseteq
\mathcal{E}(A).$

\begin{theorem}
\label{tmain4by4}Let $A\in\mathcal{PC}_{4}$ be as in (\ref{MA}), with
$a_{12},a_{13},a_{23}$ satisfying (\ref{cond11}). Then, $\mathcal{C}%
(A)\subseteq\mathcal{E}(A)$ if and only if $a_{13}\leq a_{12}a_{23}.$
\end{theorem}

\bigskip

Recall that, if $R\in\mathcal{PC}_{n}$ is a $3$-block perturbed consistent
matrix, there is a monomial matrix $S$ such that $SRS^{-1}$ has the form
(\ref{MAtriple}), with $a_{12},a_{13},a_{23}$ satisfying (\ref{cond11}). As a
consequence of Theorem \ref{tmain4by4} and Lemmas \ref{cconv} and
\ref{lemablock}, we obtain the characterization of the $3$-block perturbed
consistent matrices $R$ such that $\mathcal{C}(R)\subseteq\mathcal{E}(R)$.

\begin{theorem}
\label{tmain3block}Let $R\in\mathcal{PC}_{n}$ be a $3$-block perturbed
consistent matrix and let $A$ be a matrix as in (\ref{MAtriple}), with
$a_{12},a_{13},a_{23}$ satisfying (\ref{cond11}), monomial similar to $R$.
Then, $\mathcal{C}(R)\subseteq\mathcal{E}(R)$ if and only if $a_{13}\leq
a_{12}a_{23}.$
\end{theorem}

We then have the following result, whose claim concerning the Perron vector
was noted in \cite{FJ4}. For $n=4,$ a weaker version (not including the
boundary conditions) appeared later in \cite{Ro}, without noticing the connection.

\begin{corollary}
\label{c3blockPerron}Let $R\in\mathcal{PC}_{n}$ be a $3$-block perturbed
consistent matrix and let $A$ be a matrix as in (\ref{MAtriple}), with
$a_{12},a_{13},a_{23}$ satisfying (\ref{cond11}), monomial similar to $R$. If
$a_{13}\leq a_{12}a_{23},$ then the Perron vector and the singular vector of
$R$ are efficient for $R.$
\end{corollary}

\bigskip

\bigskip We finally note that, if $R$ is a $4$-by-$4$ double perturbed
consistent matrix, then $R$ is monomially similar to a $3$-block perturbed
consistent matrix as in (\ref{MAtriple}), with $a_{12},a_{13},a_{23}$
satisfying one of the conditions ii) or iii) in (\ref{cond11}), or both
$a_{12}=a_{13}\geq1$ and $a_{23}\geq1.$ (The latter case occurs when the two
perturbed entries above the main diagonal lie in different rows and different
columns.) Then, by Theorem \ref{tmain3block}, $\mathcal{C}(R)\subseteq
\mathcal{E}(R).$

The class of $3$-block perturbed consistent matrices of size $n>4$ also
include a certain type of double perturbed consistent matrices. However, since
such matrices are also $4$-block triangular perturbed consistent matrices, we
leave the discussion of that case for the next section.

\section{Efficiency of the convex hull of the columns of a $4$-block
triangular perturbed consistent matrix\label{s6}}

Here we give necessary and sufficient conditions for $\mathcal{C}%
(A)\subseteq\mathcal{E}(A),$ when $A$ is a $4$-block triangular perturbed
consistent matrix. By Lemmas \ref{lsim} and \ref{cconv}, we may assume that
$A$ is as in (\ref{triple}), with $a_{14}\geq1$ (by an additional permutation
similarity, if necessary)$.$ We start by considering the case $A\in
\mathcal{PC}_{5}$:
\begin{equation}
A=\left[
\begin{array}
[c]{ccccc}%
1 & 1 & a_{13} & a_{14} & 1\\
1 & 1 & 1 & a_{24} & 1\\
\frac{1}{a_{13}} & 1 & 1 & 1 & 1\\
\frac{1}{a_{14}} & \frac{1}{a_{24}} & 1 & 1 & 1\\
1 & 1 & 1 & 1 & 1
\end{array}
\right]  .\text{ } \label{MA5}%
\end{equation}

\bigskip

We use the notation $\ $introduced in Section \ref{secnotation}.

\begin{remark}
\label{Reffsub5}Let $A$ be as in (\ref{MA5}) and let $w=\left[
\begin{array}
[c]{ccccc}%
w_{1} & w_{2} & w_{3} & w_{4} & w_{5}%
\end{array}
\right]  ^{T}\in\mathcal{V}_{5}.$ According to Corollary \ref{c3por3} (Theorem
\ref{tmain} is more explicit in some cases), we have:
\end{remark}

\begin{itemize}
\item $w(1,2)$ is efficient for $A(1,2)$ if and only if
\[
w_{3}=w_{4}=w_{5}\text{; }%
\]

\item $w(1,3)$ is efficient for $A(1,3)$ if and only if
\[
w_{4}\leq w_{5}\leq w_{2}\leq a_{24}w_{4}\text{ \quad or\quad}w_{4}\geq
w_{5}\geq w_{2}\geq a_{24}w_{4}\text{;}%
\]

\item $w(1,4)$ is efficient for $A(1,4)$ if and only if
\[
w_{2}=w_{3}=w_{5}\text{;}%
\]

\item $w(1,5)$ is efficient for $A(1,5)$ if and only if
\[
w_{4}\leq w_{3}\leq w_{2}\leq a_{24}w_{4}\text{ \quad or\quad}w_{4}\geq
w_{3}\geq w_{2}\geq a_{24}w_{4};
\]

\item $w(2,3)$ is efficient for $A(2,3)$ if and only if
\[
w_{4}\leq w_{5}\leq w_{1}\leq a_{14}w_{4}\text{ \quad or\quad}w_{4}\geq
w_{5}\geq w_{1}\geq a_{14}w_{4}\text{;}%
\]

\item $w(2,4)$ is efficient for $A(2,4)$ if and only if
\[
w_{3}\leq w_{5}\leq w_{1}\leq a_{13}w_{3}\text{ \quad or\quad}w_{3}\geq
w_{5}\geq w_{1}\geq a_{13}w_{3}\text{;}%
\]

\item $w(2,5)$ is efficient for $A(2,5)$ if and only if
\[
w_{4}\leq w_{3}\leq\frac{w_{1}}{a_{13}}\leq\frac{a_{14}}{a_{13}}w_{4}\text{
\quad or\quad}w_{4}\geq w_{3}\geq\frac{w_{1}}{a_{13}}\geq\frac{a_{14}}{a_{13}%
}w_{4};
\]

\item $w(3,4)$ is efficient for $A(3,4)$ if and only if
\[
w_{1}=w_{2}=w_{5}\text{;}%
\]

\item $w(3,5)$ is efficient for $A(3,5)$ if and only if
\[
a_{24}w_{4}\leq w_{2}\leq w_{1}\leq a_{14}w_{4}\text{ \quad or\quad}%
a_{24}w_{4}\geq w_{2}\geq w_{1}\geq a_{14}w_{4}\text{;}%
\]

\item $w(4,5)$ is efficient for $A(4,5)$ if and only if
\[
w_{3}\leq w_{2}\leq w_{1}\leq a_{13}w_{3}\text{ \quad or\quad}w_{3}\geq
w_{2}\geq w_{1}\geq a_{13}w_{3}.
\]

\end{itemize}

\newpage

\begin{lemma}
\label{l5cainea}Let $A\in\mathcal{PC}_{5}$ be as in (\ref{MA5})$,$ with
$a_{14}\geq1.$ Suppose that neither of the following two conditions holds:

\begin{enumerate}
\item $1<a_{24}<a_{14}<a_{13;}$

\item $a_{24}<1<a_{13}<a_{14.}$
\end{enumerate}

Then, $\mathcal{C}(A)$ is contained in $\mathcal{E}(A).$
\end{lemma}

\begin{proof}
Let $w\in\mathcal{C}(A),$ that is, $w=Av$ for some nonnegative vector
$v=\left[
\begin{array}
[c]{ccccc}%
v_{1} & v_{2} & v_{3} & v_{4} & v_{5}%
\end{array}
\right]  ^{T}$ such that $\sum_{i=1}^{5}v_{i}=1$. Our strategy is the
following. We consider each of the following possible cases:

\begin{enumerate}
\item[i)] $1\leq a_{13},a_{24}\leq a_{14};$

\item[ii)] $1\leq a_{14}\leq a_{13},a_{24};$

\item[iii)] $1\leq a_{13}\leq a_{14}\leq a_{24};$

\item[iv)] $a_{13},a_{24}\leq1\leq a_{14};$

\item[v)] $a_{13}\leq1\leq a_{24}\leq a_{14};$

\item[vi)] $a_{13}\leq1\leq a_{14}\leq a_{24};$

\item[vii)] $a_{24}\leq1\leq a_{14}\leq a_{13}.$
\end{enumerate}

In each case i)-vii), we give $i,j,i_{1},i_{2},j_{1},j_{2}\in\{1,\ldots,n\}$
with $i,j$ distinct, $i,i_{1},i_{2}$ pairwise distinct and $j,j_{1},j_{2}$
pairwise distinct$,$ such that $w(i,i_{1})$ and $w(i,i_{2})$ are efficient for
$A(i,i_{1})$ and $A(i,i_{2})$, respectively, and $w(j,j_{1})$ and $w(j,j_{2})$
are efficient for $A(j,j_{1})$ and $A(j,j_{2})$, respectively. Then, by
Theorem \ref{tnbyn}, $w(i)$ is efficient for $A(i)$ and $w(j)$ is efficient
for $A(j).$ Again by the theorem, $w$ is efficient for $A.$ Since $w$ is an
arbitrary vector in $\mathcal{C}(A)$, this proves the claim.

We will use Remark \ref{Reffsub5} to check the efficiency for the
corresponding principal submatrix of $A$ of the vectors obtained from $w$ by
deleting two entries$.$ To facilitate the proof, we state some useful
expressions involving the entries of $w$:%

\begin{tabular}
[c]{ll}%
$w_{1}-w_{2}=v_{3}(a_{13}-1)+v_{4}(a_{14}-a_{24});$ & \\
& \\
$w_{2}-w_{3}=\frac{1}{a_{13}}\left(  v_{1}(a_{13}-1)+v_{4}a_{13}%
(a_{24}-1)\right)  ;$ & \\
& \\
$w_{3}-w_{4}=v_{1}\left(  \frac{1}{a_{13}}-\frac{1}{a_{14}}\right)
+v_{2}\left(  1-\frac{1}{a_{24}}\right)  ;$ & \\
& \\
$w_{4}-w_{5}=v_{1}\left(  \frac{1}{a_{14}}-1\right)  +v_{2}\left(  \frac
{1}{a_{24}}-1\right)  ;$ & \\
& \\
$w_{5}-w_{2}=v_{4}(1-a_{24});$ & \\
& \\
$w_{5}-w_{1}=v_{3}(1-a_{13})+v_{4}(1-a_{14});$ & \\
& \\
$w_{3}-w_{5}=v_{1}(\frac{1}{a_{13}}-1);$ & \\
& \\
$w_{1}-a_{14}w_{4}=\frac{1}{a_{24}}\left(  v_{2}(a_{24}-a_{14})+v_{3}%
a_{24}(a_{13}-a_{14})+v_{5}a_{24}(1-a_{14})\right)  ;$ & \\
& \\
$w_{1}-a_{13}w_{3}=(v_{2}+v_{5})(1-a_{13})+v_{4}(a_{14}-a_{13});$ & \\
& \\
$w_{2}-a_{24}w_{4}=\frac{1}{a_{14}}\left(  v_{1}(a_{14}-a_{24})+(v_{3}%
+v_{5})(1-a_{24})a_{14}\right)  .$ &
\end{tabular}

\bigskip

Now we consider each case (including natural subcases) and make observations
that verify the claim. To facilitate the arguments, in Table \ref{tab2} we
give the signs of the expressions above in each case, when they are uniquely
determined. By $\geq0$ ($\leq0)$ we mean a nonnegative (nonpositive) expression.

\bigskip%

\begin{table}[] \centering
\begin{tabular}
[c]{|c|c|c|c|c|c|c|c|}\hline
Case: & i) & ii) & iii) & iv) & v) & vi) & vii)\\\hline
$w_{1}-w_{2}$ & $\geq0$ &  &  &  &  & $\leq0$ & $\geq0$\\\hline
$w_{2}-w_{3}$ & $\geq0$ & $\geq0$ & $\geq0$ & $\leq0$ &  &  & \\\hline
$w_{3}-w_{4}$ & $\geq0$ &  & $\geq0$ &  & $\geq0$ & $\geq0$ & $\leq0$\\\hline
$w_{4}-w_{5}$ & $\leq0$ & $\leq0$ & $\leq0$ &  & $\leq0$ & $\leq0$ & \\\hline
$w_{5}-w_{2}$ & $\leq0$ & $\leq0$ & $\leq0$ & $\geq0$ & $\leq0$ & $\leq0$ &
$\geq0$\\\hline
$w_{5}-w_{1}$ & $\leq0$ & $\leq0$ & $\leq0$ &  &  &  & $\leq0$\\\hline
$w_{3}-w_{5}$ & $\leq0$ & $\leq0$ & $\leq0$ & $\geq0$ & $\geq0$ & $\geq0$ &
$\leq0$\\\hline
$w_{1}-a_{14}w_{4}$ & $\leq0$ &  &  & $\leq0$ & $\leq0$ &  & \\\hline
$w_{1}-a_{13}w_{3}$ &  & $\leq0$ &  & $\geq0$ & $\geq0$ & $\geq0$ & $\leq
0$\\\hline
$w_{2}-a_{24}w_{4}$ &  & $\leq0$ & $\leq0$ & $\geq0$ &  & $\leq0$ & $\geq
0$\\\hline
\end{tabular}
\caption{Signs of expressions depending on the entries of  $w$ that are relevant in the proof of Lemma \ref{l5cainea}. }\label{tab2}%
\end{table}%

\bigskip

\textbf{Case i)} Suppose that $1\leq a_{13},a_{24}\leq a_{14}.$ By Remark
\ref{Reffsub5}, $w(2,3)$ is efficient for $A(2,3).$

\begin{itemize}
\item Suppose that $w_{2}\leq a_{24}w_{4}.$ Then $w(1,3)$ is efficient for
$A(1,3)$ and $w(1,5)$ is efficient for $A(1,5).$ Thus, $w(1)$ is efficient for
$A(1)$ and $w(3)$ is efficient for $A(3)$.

\item Suppose that $w_{2}\geq a_{24}w_{4}.$ Then $w(3,5)$ is efficient for
$A(3,5).$ If $a_{13}w_{3}\leq w_{1},$ then $w(2,5)$ is efficient for $A(2,5),$
else $w(2,4)$ is efficient for $A(2,4).$ In any case, $w(2)$ is efficient for
$A(2)$ and $w(3)$ is efficient for $A(3).$
\end{itemize}

\bigskip\newpage

\textbf{Case ii)} Suppose that $1\leq a_{14}\leq a_{13},a_{24}.$ By Remark
\ref{Reffsub5}, $w(1,3)$ is efficient for $A(1,3)$ and $w(2,4)$ is efficient
for $A(2,4).$

\begin{itemize}
\item Suppose that $w_{1}\leq a_{14}w_{4}.$ Then $w(2,3)$ is efficient for
$A(2,3).$ Thus, $w(2)$ is efficient for $A(2)$ and $w(3)$ is efficient for
$A(3).\ $

\item Suppose that $w_{1}\geq a_{14}w_{4}.$ If $w_{2}\leq w_{1}$ then $w(4,5)$
is efficient for $A(4,5),$ else $w(3,5)$ is efficient for $A(3,5)$.

\begin{itemize}
\item If $w_{4}\leq w_{3}$, then $w(1,5)$ is efficient for $A(1,5).$ Thus,
$w(1)$ is efficient for $A(1)$ and $w(5)$ is efficient for $A(5).$

\item If $w_{4}\geq w_{3}$, then $w(2,5)$ is efficient for $A(2,5).$ Thus,
$w(2)$ is efficient for $A(2)$ and $w(5)$ is efficient for $A(5).$
\end{itemize}
\end{itemize}

\bigskip

\textbf{Case iii)} Suppose that $1\leq a_{13}\leq a_{14}\leq a_{24}.$ By
Remark \ref{Reffsub5}, $w(1,3)$ is efficient for $A(1,3)$ and $w(1,5)$ is
efficient for $A(1,5).$ Thus, $w(1)$ is efficient for $A(1)$.

\begin{itemize}
\item Suppose that $w_{1}\leq a_{14}w_{4}.$ Then $w(2,3)$ is efficient for
$A(2,3).$ Thus, $w(1)$ is efficient for $A(1)$ and $w(3)$ is efficient for
$A(3).$

\item Suppose that $w_{1}\geq a_{14}w_{4}.$ First we show that either
$w_{2}\geq w_{1}$ or $w_{1}\leq a_{13}w_{3}.$ Suppose that $w_{2}<w_{1}$ and
$w_{1}>a_{13}w_{3}.$ The former implies $a_{13}>1$ and then, from $w_{1}\geq
a_{14}w_{4},$ we get $a_{24}-a_{14}>0.$ We have
\begin{equation}
w_{2}<w_{1}\Leftrightarrow v_{3}>\frac{v_{4}(a_{24}-a_{14})}{a_{13}%
-1},\label{condv3}%
\end{equation}%
\begin{align*}
w_{1} &  \geq a_{14}w_{4}\Leftrightarrow v_{2}\geq\frac{v_{3}a_{24}%
(a_{14}-a_{13})+v_{5}a_{24}(a_{14}-1)}{a_{24}-a_{14}}\\
&  \Rightarrow v_{2}\geq\frac{v_{3}a_{24}(a_{14}-a_{13})}{a_{24}-a_{14}},
\end{align*}
and%
\begin{align*}
w_{1} &  >a_{13}w_{3}\Leftrightarrow v_{2}<\frac{v_{4}(a_{14}-a_{13}%
)+v_{5}(1-a_{13})}{a_{13}-1}\\
&  \Rightarrow v_{2}<\frac{v_{4}(a_{14}-a_{13})}{a_{13}-1}.
\end{align*}
Then,
\[
\frac{v_{3}a_{24}(a_{14}-a_{13})}{a_{24}-a_{14}}<\frac{v_{4}(a_{14}-a_{13}%
)}{a_{13}-1}\Rightarrow v_{3}<\frac{v_{4}(a_{24}-a_{14})}{a_{24}(a_{13}-1)},
\]
contradicting (\ref{condv3}). So, we just need to consider the following subcases.

\begin{itemize}
\item Suppose that $w_{2}\geq w_{1}.$ Then, $w(3,5)$ is efficient for
$A(3,5).$ Thus, $w(1)$ is efficient for $A(1)$ and $w(5)$ is efficient for
$A(5).$

\item Suppose that $w_{2}\leq w_{1}$ and $w_{1}\leq a_{13}w_{3}.$ Then,
$w(4,5)$ is efficient for $A(4,5).$ Thus, $w(1)$ is efficient for $A(1)$ and
$w(5)$ is efficient for $A(5).$
\end{itemize}
\end{itemize}

\bigskip

\textbf{Case iv)} Suppose that $a_{13},a_{24}\leq1\leq a_{14}.$ If $w_{4}\geq
w_{3}$ then $w(1,5)$ is efficient for $A(1,5),$ else $w(2,5)$ is efficient for
$A(2,5).$ If $w_{2}\geq w_{1}$ then $w(4,5)$ is efficient for $A(4,5),$ else
$w(3,5)$ is efficient for $A(3,5).$ Thus, $w(5)$ is efficient for $A(5).$

If $w_{4}\geq w_{5}$ then $w(1,3)$ is efficient for $A(1,3).$ If $w_{5}\geq
w_{1}$ then $w(2,4)$ is efficient for $A(2,4)$. If $w_{4}\leq w_{5}$ and
$w_{5}\leq w_{1}$ then $w(2,3)$ is efficient for $A(2,3)$. We also note that
$w_{3}\geq w_{5}\geq w_{2}.$ So, we consider the following subcases.

\begin{itemize}
\item If $w_{4}\geq w_{5},$ $w_{2}\geq w_{1}$ and $w_{4}\geq w_{3},$ then
$w(1)$ is efficient for $A(1)$ and $w(5)$ is efficient for $A(5).$

\item If $w_{4}\geq w_{5},$ $w_{2}\geq w_{1}$ and $w_{4}\leq w_{3},$ then
$\ w_{5}\geq w_{1},$ as $w_{5}\geq w_{2}.$ Thus, $w(2,4)$ is efficient for
$A(2,4).$ Then, $w(2)$ is efficient for $A(2)$ and $w(5)$ is efficient for
$A(5).$

\item If $w_{4}\geq w_{5}$, $w_{2}\leq w_{1},$ then $w(3)$ is efficient for
$A(3)$ and $w(5)$ is efficient for $A(5).$

\item If $w_{4}\leq w_{5}$, $w_{4}\leq w_{3}$ and $w_{5}\geq w_{1},$ then
$w(2)$ is efficient for $A(2)$ and $w(5)$ is efficient for $A(5).$

\item If $w_{4}\leq w_{5}$, $w_{4}\leq w_{3},$ $w_{5}\leq w_{1}$ and
$w_{2}\leq w_{1},$ then $w(3)$ is efficient for $A(3)$ and $w(5)$ is efficient
for $A(5).$
\end{itemize}

\bigskip

\textbf{Case v)} Suppose that $a_{13}\leq1\leq a_{24}\leq a_{14}.$ By Remark
\ref{Reffsub5}, $w(2,5)$ is efficient for $A(2,5).$ If $w_{5}\leq w_{1}$ then
$w(2,3)$ is efficient for $A(2,3),$ otherwise $w(2,4)$ is efficient for
$A(2,4).$ In any case, $w(2)$ is efficient for $A(2).$

\begin{itemize}
\item Suppose that $w_{2}\leq w_{1}.$ Then $w_{5}\leq w_{1},$ as $w_{5}\leq
w_{2}$, and $w(2,3)$ is efficient for $A(2,3).$ If $w_{2}\leq a_{24}w_{4},$
then $w(1,3)$ is efficient for $A(1,3),$ otherwise $w(3,5)$ is efficient for
$A(3,5).$ In any case, $w(2)$ is efficient for $A(2)$ and $w(3)$ is efficient
for $A(3).$

\item If $w_{2}\geq w_{1}$ and $w_{3}\geq w_{2}$, then $w(4,5)$ is efficient
for $A(4,5).$ Thus, $w(2)$ is efficient for $A(2)$ and $w(5)$ is efficient for
$A(5).$

\item Suppose that $w_{2}\geq w_{1}$ and $w_{3}\leq w_{2}.$ We show next that
$w_{2}\leq a_{24}w_{4}.$ In this case, $w(1,3)$ is efficient for $A(1,3)\ $and
$w(1,5)$ is efficient for $A(1,5).$ Then $w(1)$ is efficient for $A(1)$ and
$w(2)$ is efficient for $A(2).$ Now suppose that $w_{2}>a_{24}w_{4},$ in order
to get a contradiction. Then, $a_{14}-a_{24}>0$ and $a_{13}<1.$ Thus,%
\begin{align*}
w_{2}  &  >a_{24}w_{4}\Leftrightarrow v_{1}>\frac{(v_{3}+v_{5})(a_{24}%
-1)a_{14}}{a_{14}-a_{24}}\\
w_{2}  &  \geq w_{1}\Leftrightarrow v_{3}\geq\frac{v_{4}(a_{14}-a_{24}%
)}{1-a_{13}}\\
w_{3}  &  \leq w_{2}\Leftrightarrow v_{1}\leq\frac{v_{4}a_{13}(a_{24}%
-1)}{1-a_{13}}.
\end{align*}
This implies
\begin{align*}
\frac{(v_{3}+v_{5})(a_{24}-1)a_{14}}{a_{14}-a_{24}}  &  <\frac{v_{4}%
a_{13}(a_{24}-1)}{1-a_{13}}\Rightarrow v_{3}+v_{5}<\frac{v_{4}a_{13}%
(a_{14}-a_{24})}{a_{14}(1-a_{13})}\\
&  \Rightarrow v_{3}<\frac{v_{4}a_{13}(a_{14}-a_{24})}{a_{14}(1-a_{13})}.
\end{align*}
Then%
\[
\frac{v_{4}(a_{14}-a_{24})}{1-a_{13}}<\frac{v_{4}a_{13}(a_{14}-a_{24})}%
{a_{14}(1-a_{13})}\Leftrightarrow a_{14}<a_{13},
\]
a contradiction.
\end{itemize}

\bigskip

\textbf{Case vi)} Suppose that $a_{13}\leq1\leq a_{14}\leq a_{24}.$ We have
that $w(1,3)$ is efficient for $A(1,3).$

\begin{itemize}
\item Suppose that $w_{1}\leq a_{14}w_{4}.$ We have that $w(2,5)$ is efficient
for $A(2,5).$ If $w_{5}\leq w_{1}$ then $w(2,3)$ is efficient for $A(2,3),$
otherwise $w(2,4)$ is efficient for $A(2,4).$ If $w_{3}\leq w_{2}$ then
$w(1,5)$ is efficient for $A(1,5),$ otherwise $w(4,5)$ is efficient for
$A(4,5).$ In all situations, $w(2)$ is efficient for $A(2)$ and $w(5)$ is
efficient for $A(5).$

\item Suppose that $w_{1}\geq a_{14}w_{4}.$ We have that $w(3,5)$ is efficient
for $A(3,5).$ If $w_{3}\leq w_{2}$ then $w(1,5)$ is efficient for $A(1,5),$
otherwise $w(4,5)$ is efficient for $A(4,5).$ In all cases, $w(3)$ is
efficient for $A(3)$ and $w(5)$ is efficient for $A(5).$
\end{itemize}

\bigskip

\textbf{Case vii)} Suppose that $a_{24}\leq1\leq a_{14}\leq a_{13}.$ We have
that $w(2,4)$ is efficient for $A(2,4).$

\begin{itemize}
\item Suppose that $w_{1}\leq a_{14}w_{4}.$ We have that $w(3,5)$ is efficient
for $A(3,5).$

\begin{itemize}
\item If $w_{4}\leq w_{5}$ then $w(2,3)$ is efficient for $A(2,3)$.\ Thus,
$w(2)$ is efficient for $A(2)$ and $w(3)$ is efficient for $A(3).$

\item If $w_{4}\geq w_{5}$ then $w(1,3)$ is efficient for $A(1,3).$ If
$w_{3}\leq w_{2}$ then $w(4,5)$ is efficient for $A(4,5),$ otherwise $w(1,5)$
is efficient for $A(1,5).$ In all cases, $w(3)$ is efficient for $A(3)$ and
$w(5)$ is efficient for $A(5).$
\end{itemize}

\item Suppose that $w_{1}\geq a_{14}w_{4}.$ We have that $w(2,5)$ is efficient
for $A(2,5).$ If $w_{3}\leq w_{2}$ then $w(4,5)$ is efficient for $A(4,5),$
otherwise $w(1,5)$ is efficient for $A(1,5).$ In both cases, $w(2)$ is
efficient for $A(2)$ and $w(5)$ is efficient for $A(5).$
\end{itemize}
\end{proof}

Next we show that, when $A\in\mathcal{PC}_{5}$ is as in (\ref{MA5}), with
$a_{14}\geq1,$ the sufficient conditions in Lemma \ref{l5cainea} for
$\mathcal{C}(A)$ to be contained in $\mathcal{E}(A)$ are also necessary.

\begin{lemma}
\label{l5caineaconv}Let $A\in\mathcal{PC}_{5}$ be as in (\ref{MA5})$.$ Suppose
that one of the following conditions holds:

\begin{enumerate}
\item $1<a_{24}<a_{14}<a_{13},$ or

\item $a_{24}<1<a_{13}<a_{14}.$
\end{enumerate}

Then, $\mathcal{C}(A)$ is not contained in $\mathcal{E}(A).$
\end{lemma}

\begin{proof}
In each of the two cases, we give a vector in $\mathcal{C}(A)$ that is not in
$\mathcal{E}(A)$.

\bigskip

\textbf{Case i)} Suppose that $1<a_{24}<a_{14}<a_{13}.$ Let $\varepsilon
_{1},\varepsilon_{2}>0$ and let
\begin{align*}
u_{1}  &  =u_{4}=1;\text{ }u_{5}=0;\\
u_{2}  &  =\frac{a_{24}(a_{13}-a_{14})}{a_{13}a_{14}(a_{24}-1)}+\varepsilon
_{1};\\
u_{3}  &  =\frac{a_{14}-a_{24}}{a_{13}a_{14}(a_{24}-1)}+\varepsilon_{2}.
\end{align*}
Let $w=Au,$ with $u=\left[
\begin{array}
[c]{ccccc}%
u_{1} & u_{2} & u_{3} & u_{4} & u_{5}%
\end{array}
\right]  ^{T}$. Since $u$ is a nonzero nonnegative vector, $w\in
\mathcal{C}(A).$ We have
\[
w_{3}-w_{4}=\frac{a_{24}-1}{a_{24}}\varepsilon_{1}>0.
\]
Also, for $\varepsilon_{2}$ sufficiently small, we have
\[
w_{2}-a_{24}w_{4}=\frac{\left(  a_{13}-1\right)  \left(  a_{14}-a_{24}\right)
}{a_{13}a_{14}}+\left(  1-a_{24}\right)  \varepsilon_{2}>0.
\]
In addition, for $\varepsilon_{1}$ sufficiently small, we have
\[
w_{1}-a_{14}w_{4}=\frac{a_{24}-a_{14}}{a_{24}}\varepsilon_{1}+(a_{13}%
-a_{14})\varepsilon_{2}>0.
\]
Since%
\[
w_{1}-w_{3}=\left(  a_{14}-1\right)  \left(  1+\frac{a_{24}(a_{13}-1)}%
{a_{13}a_{14}\left(  a_{24}-1\right)  }\right)  +\left(  a_{13}-1\right)
\varepsilon_{2}>0,
\]
$w_{1}$ and $w_{3}$ are distinct. By Remark \ref{Reffsub5}, if $w(i,j)$ is
efficient for $A(i,j),$ $i<j,$ then $(i,j)\in\{(1,4),(2,4),(3,4),(4,5)\}.$
Thus, by Theorem \ref{tnbyn}, for $\varepsilon_{1},\varepsilon_{2}$
sufficiently small, $w$ is inefficient for $A.$

\bigskip

\textbf{Case ii)} Suppose that $a_{24}<1<a_{13}<a_{14}.$ Let $\varepsilon
_{1},\varepsilon_{2}>0$ and let
\begin{align*}
u_{3}  &  =u_{4}=1;\text{ }u_{5}=0;\\
u_{1}  &  =\frac{a_{13}\left(  1-a_{24}\right)  }{a_{13}-1}+\varepsilon_{1};\\
u_{2}  &  =\frac{a_{24}(a_{14}-a_{13})}{a_{14}(a_{13}-1)}+\varepsilon_{2}.
\end{align*}
Let $w=Au,$ with $u=\left[
\begin{array}
[c]{ccccc}%
u_{1} & u_{2} & u_{3} & u_{4} & u_{5}%
\end{array}
\right]  ^{T}$. Since $u$ is a nonzero nonnegative vector, $w\in
\mathcal{C}(A).$ We have
\[
w_{2}-w_{3}=\frac{a_{13}-1}{a_{13}}\varepsilon_{1}>0.
\]
For $\varepsilon_{2}$ sufficiently small, we have%
\[
w_{1}-a_{13}w_{3}=\frac{(a_{14}-a_{13})(a_{14}-a_{24})}{a_{14}}+(1-a_{13}%
)\varepsilon_{2}>0,
\]
and, in addition, for $\varepsilon_{1}$ sufficiently small,
\[
w_{4}-w_{3}=\frac{a_{13}-a_{14}}{a_{13}a_{14}}\varepsilon_{1}+\frac{1-a_{24}%
}{a_{24}}\varepsilon_{2}>0.
\]
Since
\[
w_{5}-w_{3}=1-a_{24}+\frac{a_{13}-1}{a_{13}}\varepsilon_{1}>0,
\]
$w_{3}$ and $w_{5}$ are distinct. By Remark \ref{Reffsub5}, if $w(i,j)$ is
efficient for $A(i,j),$ $i<j,$ then $(i,j)\in\{(1,3),(2,3),(3,4),(3,5)\}.$
Thus, by Theorem \ref{tnbyn}, for $\varepsilon_{1},\varepsilon_{2}$
sufficiently small, $w$ is inefficient for $A.$
\end{proof}

\bigskip

We then present the main results of this section. From Lemmas \ref{l5cainea}
and \ref{l5caineaconv} we have the characterization of the matrices
$A\in\mathcal{PC}_{5}$ as in (\ref{MA5}), with $a_{14}\geq1,$ for which
$\mathcal{C}(A)\subseteq\mathcal{E}(A).$

\begin{theorem}
\label{tmain5by5}Let $A\in\mathcal{PC}_{5}$ be as in (\ref{MA5}) with
$a_{14}\geq1$. Then, $\mathcal{C}(A)\subseteq\mathcal{E}(A)$ if and only if
neither of the conditions 1. or 2. in Lemma \ref{l5cainea} holds.
\end{theorem}

\bigskip

Recall that, if $R\in\mathcal{PC}_{n}$ is a $4$-block triangular perturbed
consistent matrix, there is a monomial matrix $S$ such that $SRS^{-1}$ has the
form (\ref{triple}), with $a_{14}\geq1.$ As a consequence of Theorem
\ref{tmain5by5} and Lemmas \ref{cconv} and \ref{lemablock}, we obtain the
characterization of the $4$-block triangular perturbed consistent matrices $R$
such that $\mathcal{C}(R)\subseteq\mathcal{E}(R)$.

\begin{theorem}
\label{thtriangular}Let $R\in\mathcal{PC}_{n}$ be a $4$-block triangular
perturbed consistent matrix and let $A$ be a matrix as in (\ref{triple}), with
$a_{14}\geq1$, monomially similar to $R$. Then, $\mathcal{C}(R)\subseteq
\mathcal{E}(R)$ if and only if neither of conditions 1. or 2. in Lemma
\ref{l5cainea} holds$.$
\end{theorem}

If $R$ is a double perturbed consistent matrix of size $n\geq5$, then $R$ is
$\ $a $4$-block triangular perturbed consistent matrix. More precisely, $R$ is
monomially similar to a matrix of the form (\ref{triple}), with $a_{14}\geq1$
and one of $a_{13},$ $a_{14}$ or $a_{24}$ equal to $1.$ All double perturbed
consistent matrices of size $4$ are $3$-block perturbed consistent matrices$.$
Then, by Theorem \ref{thtriangular} and the comments at the end of Section
\ref{s4}, we have the following (when $n=3$ the matrix is a simple perturbed
consistent matrix and the result is covered by Proposition \ref{prop3}).

\begin{corollary}
If $R\in\mathcal{PC}_{n}$ is a double perturbed consistent matrix, then
$\mathcal{C}(R)\subseteq\mathcal{E}(R).$
\end{corollary}

\bigskip

We also have the following consequence of Theorem \ref{thtriangular}.

\begin{corollary}
\label{corPerronsimple}Let $R\in\mathcal{PC}_{n}$ be a $4$-block triangular
perturbed consistent matrix and let $A$ be a matrix as in (\ref{triple}), with
$a_{14}\geq1$, monomially similar to $R$. If neither of conditions 1. or 2. in
Lemma \ref{l5cainea} holds$,$ then the Perron vector and the singular vector
of $R$ are efficient for $R.$
\end{corollary}

Sufficient conditions for the efficiency of the Perron vector of a $4$-block
triangular perturbed consistent matrix were first obtained in \cite{FerFur}.
In \cite{Ro} other conditions were obtained in a similar way. In fact, these
known conditions are the sufficient conditions provided in Corollary
\ref{corPerronsimple}, except that they do not include the boundary matrices
given in our result. In particular, they do not cover some cases of double
perturbed consistent matrices. Our conditions for the efficiency of the Perron
vector are necessary and sufficient for the efficiency of all the cone
generated by the columns.

\bigskip

\section{Numerical experiments\label{s7}}

Here, we consider reciprocal matrices $A\in\mathcal{PC}_{n}$\ obtained from a
consistent matrix by modifying $3$ pairs of reciprocal entries located in a
$4$-by-$4$ principal submatrix, as in Sections \ref{s4} and \ref{s6} (and for
which the convex hull of the columns is contained in $\mathcal{E}(A)$)$.$ For
such matrices, we compare the behavior of convex combinations of the columns
and weighted geometric means of the columns.

We generate $\ $random vectors $\alpha=(\alpha_{1},\ldots,\alpha_{n})$ with
nonnegative entries summing to $1.$ Letting $a_{i}$ be the $i$th column of
$A,$ we consider the convex combinations of the columns of $A$
\[
\alpha_{1}a_{1}+\cdots+\alpha_{n}a_{n},
\]
as well as the weighted geometric means of the columns
\[
a_{1}^{(\alpha_{1})}\circ\cdots\circ a_{n}^{(\alpha_{n})}.
\]
For each vector $w$, we take $\left\Vert M(A,w)\right\Vert _{2},$ the
Frobenius norm of $M(A,w)=ww^{(-T)}-A,$ $\ $ as a measure of effectiveness of
$w\in\mathcal{E}(A).$ Here, by $w^{(-T)}$ we denote the transpose of the
entry-wise inverse of $w.$ Recall that, for an $n$-by-$n$ real matrix
$B=[b_{ij}],$ we have%
\[
\left\Vert B\right\Vert _{2}=\left(  \sum_{i,j=1,\ldots,n}(b_{ij})^{2}\right)
^{\frac{1}{2}}.
\]

\bigskip

We will see that, with a similarity on the reciprocal matrix, by a positive
diagonal matrix, the relative performance of the convex combinations of the
columns and of the weighted geometric means of the columns may change. These
different behaviors emphasize the importance of having a large class of
efficient vectors from which a good weight vector can be chosen, for the
particular reciprocal matrix obtained in a practical problem.

\bigskip

As a second part of our experiments, we consider matrices of the same type as
before but for which the convex hull of the columns is not contained in
$\mathcal{E}(A)$. We generate random convex combinations of the columns and
see how often they are inefficient.

\bigskip

We present data for a $3$-block perturbed consistent matrix, as in Section
\ref{s4}, since the data for  $4$-block triangular perturbed consistent matrices
were similar. Our experiments were done using the software MATLAB version R2023b.

\bigskip

We note that the matrix $M(A,w)$ is helpful in recognizing an $s$-block
perturbed consistent matrix. In fact, $M(A,w)$ for a certain column $w$ of $A$
should leave $0$'s outside the principal submatrix corresponding to the
smallest $s$-by-$s$ block that is perturbed.

\bigskip

\begin{example}
\label{ex1}Consider the $3$-block perturbed consistent matrix
\[
A=\left[
\begin{tabular}
[c]{c|c}%
$%
\begin{array}
[c]{ccc}%
1 & 4 & 3\\
\frac{1}{4} & 1 & 2\\
\frac{1}{3} & \frac{1}{2} & 1
\end{array}
$ & $J_{3,5}$\\\hline
$J_{5,3}$ & $J_{5}$%
\end{tabular}
\ \ \right]  \in\mathcal{PC}_{8}%
\]
and the positive diagonal matrices
\begin{align*}
D_{1}  &  =\operatorname*{diag}(1.5,\text{ }4,\text{ }0.5)\oplus I_{5},\\
D_{2}  &  =\operatorname*{diag}(0.3,\text{ }0.4,\text{ }0.7)\oplus I_{5}.
\end{align*}
Note that, by Theorem \ref{tmain3block}, $\mathcal{C}(A)\subseteq
\mathcal{E}(A).$ For each of the $3$-block perturbed consistent matrices $A,$
$D_{1}^{-1}AD_{1}$ and $D_{2}^{-1}AD_{2},$ we consider $100$ trials and in
each trial $i$ generate a vector $\alpha_{(i)}$ with $8$ nonnegative entries
from a uniform distribution on $(0,1)$ and normalize it to have entries
summing to $1.$ In Figure \ref{fig1I} we provide the Frobenius norm of
$M(A,w_{(i)}),$ $i=1,\ldots,100,$ when $w_{(i)}$ is the the convex combination
of the columns of $A$ with coefficients given by $\alpha_{(i)},$ denoted by
$wa_{(i)},$ and when $w_{(i)}$ is the weighted geometric mean of the columns
of $A$ with weights given by $\alpha_{(i)},$ denoted by $wg_{(i)}.$ In the $x$
axis we have the different indices $i$ corresponding to a vector $\alpha
_{(i)}.$ In the $y$ axis we have the values of $\left\Vert M(A,w_{(i)}%
)\right\Vert _{2}$. A line jointing the values of these norms in each trial is
plotted. Horizontal lines corresponding to the considered norms for the
geometric mean of all columns, $w_{gm},$ for the Perron vector, $w_{P},$ for
the singular vector, $w_{s},$ and for the arithmetic mean of the columns (or,
analogously, the convex combination of the columns with constant
coefficients), $w_{sum},$ also appear.

In Figures \ref{figD1} and \ref{figD2} we present analogous graphics for the
matrices $D_{1}^{-1}AD_{1}$ and $D_{2}^{-1}AD_{2},$ respectively.

It can be observed that when a diagonal similarity via $D_{1}$ is considered,
the convex combinations of the columns perform better than the weighted
geometric means, though the latter perform better if we consider the
similarity via $D_{2}$. The relative performance of the the geometric and
arithmetic means of the columns, the Perron vector and the singular vector
also varies. The singular vector seems to have an extreme behavior (is almost
the best when the convex combinations of the columns are better than the
weighted geometric means, and almost the worst otherwise).

According to Theorem \ref{tmain3block}, if $A\in\mathcal{PC}_{n}$ is a
$3$-block perturbed consistent matrix as in (\ref{MAtriple}), with $a_{12}=4$
and $a_{23}=2,$ and $a_{13}>2\times4=8$, then the convex hull of the columns
is not contained in the set of efficient vectors for $A$. In Table \ref{tab3}
we give the number of inefficient vectors among $10000$ random convex
combinations of the columns, generated as above, for some values of $n$ and
$a_{13}>8.$ It can be observed that the number of inefficient vectors
increases when $a_{13}$ increases and is not too large, and then stabilizes.
In all cases, the singular vector and the arithmetic mean of all columns are
efficient. In Table \ref{tab31} we indicate the situations in which the Perron
vector is inefficient.
\end{example}

\bigskip%
\begin{table}[] \centering
$%
\begin{tabular}
[c]{|c|c|c|c|c|c|c|c|c|}\hline
$a_{13}$ & $8.2$ & $9$ & $12$ & $20$ & $50$ & $100$ & $1000$ & $10000$\\\hline
\multicolumn{1}{|l|}{$n=4$} & $60$ & $250$ & $1098$ & $2205$ & $2913$ & $3058$
& $3150$ & $3099$\\\hline
\multicolumn{1}{|l|}{$n=8$} & $0$ & $0$ & $2$ & $249$ & $2170$ & $2978$ &
$3605$ & $3684$\\\hline
\multicolumn{1}{|l|}{$n=20$} & $0$ & $0$ & $0$ & $0$ & $167$ & $1476$ & $3553$
& $3699$\\\hline
\multicolumn{1}{|l|}{$n=100$} & $0$ & $0$ & $0$ & $0$ & $0$ & $0$ & $2369$ &
$3553$\\\hline
\end{tabular}
\ \ \ \ \ \ \ \ \ $%
\caption{Number of inefficient vectors among $10000$ random convex combinations of the columns of $A$ as in (\ref{MAtriple}), with $a_{12}=4$ and $a_{23}=2$, for several values of $a_{13}$ for which the convex hull of the columns of  A is not contained in $\mathcal{E}(A)$. }\label{tab3}%
\end{table}%

\bigskip%
\begin{table}[] \centering
$%
\begin{tabular}
[c]{|c|c|c|c|c|c|c|c|c|}\hline
$a_{13}$ & $8.2$ & $9$ & $12$ & $20$ & $50$ & $100$ & $1000$ & $10000$\\\hline
\multicolumn{1}{|l|}{$n=4$} &  &  &  & x & x &  &  & \\\hline
\multicolumn{1}{|l|}{$n=8$} &  &  &  &  & x & x & x & \\\hline
\multicolumn{1}{|l|}{$n=20$} &  &  &  &  &  & x & x & x\\\hline
\multicolumn{1}{|l|}{$n=100$} &  &  &  &  &  &  & x & x\\\hline
\end{tabular}
\ \ \ \ \ \ \ \ \ $%
\caption{Inefficiency of the Perron vector  of $A$ as in (\ref{MAtriple}), with $a_{12}=4$ and $a_{23}=2$, for several values of $a_{13}$ for which the convex hull of the columns of  A is not contained in $\mathcal{E}(A)$. }\label{tab31}%
\end{table}%
%

\begin{figure}[htb]
 \includegraphics[width=\linewidth]{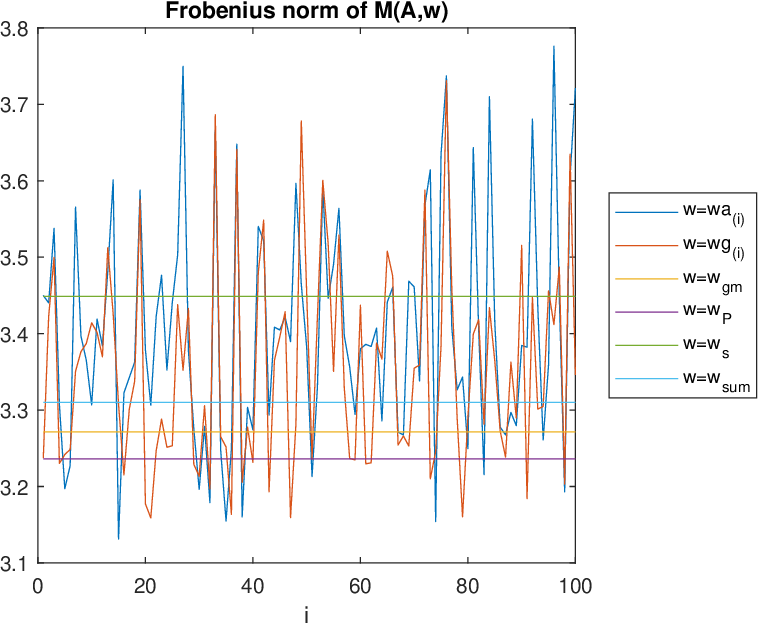}
 \caption{\label{fig1I} Comparison of the convex combinations and the weighted geometric
means of the columns of $A$ (Example \ref{ex1}).}

\end{figure}
%

\begin{figure}[htb]
 \includegraphics[width=\linewidth]{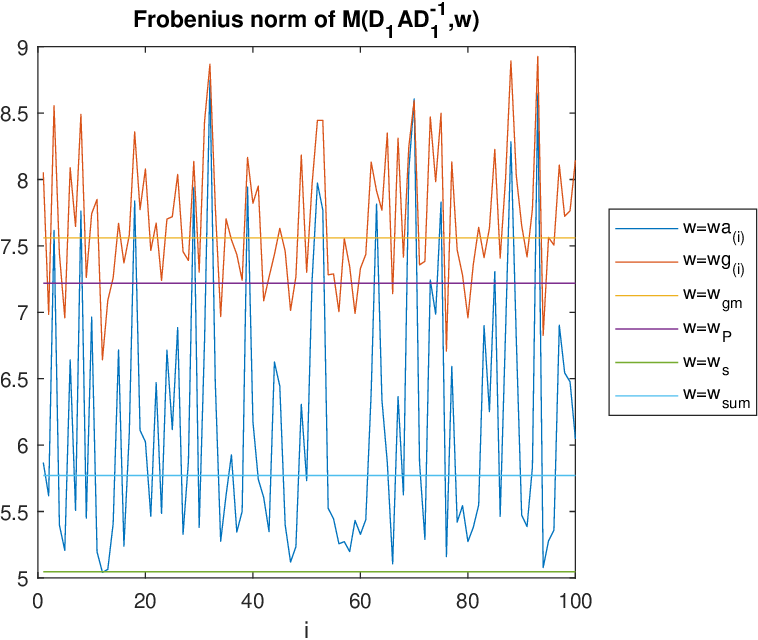}
 \caption{\label{figD1} Comparison of the convex combinations and the weighted geometric
means of the columns of $D_{1}^{-1}AD_{1}$ (Example \ref{ex1}).}

 \end{figure}

%


\begin{figure}[htb]
 \includegraphics[width=\linewidth]{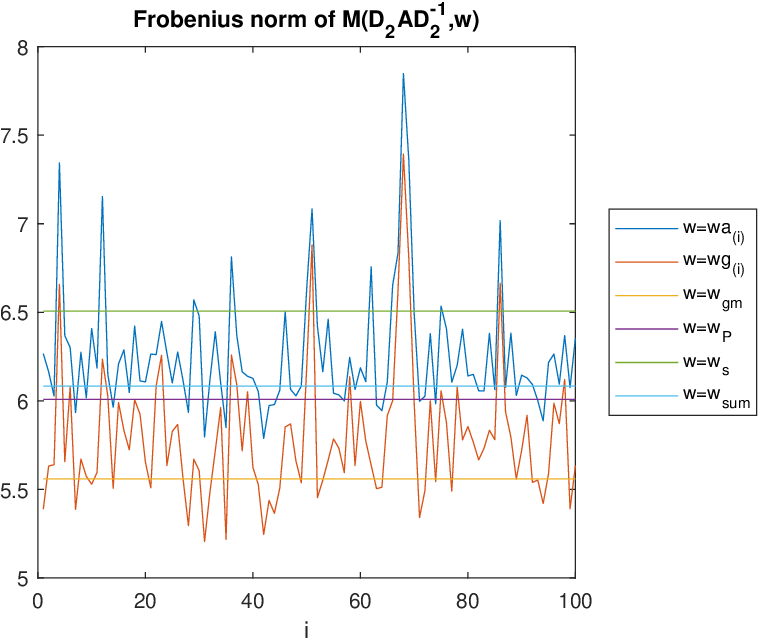}
 \caption{\label{figD2} Comparison of the convex combinations and the weighted geometric
means of the columns of $D_{2}^{-1}AD_{2}$ (Example \ref{ex1}).}

 \end{figure}

\section{Conclusions\label{s8}}

For reciprocal matrices obtained from consistent matrices by modifying $3$
pairs of reciprocal entries contained in a $4$-by-$4$ principal submatrix,
sufficient conditions were known for the efficiency of the Perron vector. For
such matrices, we show that these conditions are necessary and sufficient for
the convex hull of the columns to be contained in the set of efficient vectors
for the matrix. Since the (normalized) Perron vector of a matrix $A$ is in the
convex hull of its columns, our new results are much more general. When the
stated conditions hold, they provide new efficient vectors, in particular the
Perron vector of $AA^{T}$ (the singular vector of $A$). When there are convex
combinations of the columns of $A$ that are inefficient for $A$ then the set
of efficient vectors for $A$ is not convex, as any column of $A$ is efficient.

Numerical experiments for matrices in the classes studied compared vectors in
the convex hull of the columns with weighted geometric means of the columns
(known to be efficient). It can be seen that their relative performance may
change if a diagonal similarity is applied to the matrix. This emphasizes the
importance in practice of having a large class of efficient vectors from which
a weight vector can be chosen.

Whether the convex hull of the columns of a reciprocal matrix $A$ is contained
in the set of efficient vectors for $A$ is an important question for which
there is no general answer. This containment ensures the efficiency of the
Perron vector of $A,$ a classical proposal for the weight vector obtained from
a reciprocal matrix, and provides new simple efficient vectors obtained from
the columns of the matrix (in particular the recently proposed singular vector
of $A$)$.$

\section*{Disclosure of interest}

\bigskip

We declare there are no competing interests.

\end{document}